\documentclass[12pt,reqno]{amsart}

\usepackage{color}
\usepackage{amsmath, amsthm, amssymb, amsxtra}
\usepackage{amsfonts}
\usepackage[utf8]{inputenc}
\usepackage{graphicx}
\usepackage[english]{babel}
\usepackage{hyperref}
\usepackage{pgfplots}
\pgfplotsset{compat=1.18}

\usepackage{latexsym}
\usepackage{mathtools}
\usepackage{tikz}
\usepackage{esint}
\usepackage{float}

\newcommand{\R}{\mathbb{R}}

\theoremstyle{plain}
\newtheorem{teor}{Theorem}[section]

\newtheorem{lema}[teor]{Lemma}
\newtheorem{prop}[teor]{Proposition}

\theoremstyle{definition}
\newtheorem{defi}[teor]{Definition}

\theoremstyle{remark}
\newtheorem{nota}[teor]{Remark}

\numberwithin{equation}{section}
\renewcommand{\theequation}{\thesection.\arabic{equation}}

\newcommand{\dx}[0]{\,\mathrm{d}x}

\title{An improved stability result for Grünbaum's inequality}
\author{Luca Tanganelli Castrillón}
\thanks{This work was conducted within the scope of the program ``Barcelona Introduction to Mathematical Research'', with the generous support of Barcelona's Centre de Recerca Matemàtica. I especially thank Clara Torres Latorre for her dedicated guidance throughout the writing of this article, as well as Grégory Miermont for his insightful comments.}
\date{September 2024}
\begin{document}

\begin{abstract}
    Given a hyperplane $H$ cutting a compact, convex body $K$ of positive Lebesgue measure through its centroid, Grünbaum proved that
    $$\frac{|K\cap H^+|}{|K|}\geq \left(\frac{n}{n+1}\right)^n,$$where $H^+$ is a half-space of boundary $H$. The inequality is sharp and equality is reached only if $K$ is a cone. Moreover, bodies that almost achieve equality are geometrically close to being cones, as Groemer showed in 2000 by giving his stability estimates for Grünbaum's inequality. In this paper, we improve the exponent in the stability inequality from Groemer's $\frac{1}{2n^2}$ to $\frac{1}{2n}$.
\end{abstract}

\maketitle

\section{Introduction}

An inequality of Grünbaum's (\cite{Gru1}) states that, given a compact, convex body $K\subset\R^n$, $n\in \mathbb{N}^*$, of positive Lebesgue measure and a hyperplane $H$ cutting through the centroid of $K$, we have
\begin{equation}\label{grunbaum}
    \frac{|K\cap H^+|}{|K|}\geq \left(\frac n{n+1}\right)^n=:q_n
\end{equation}
where $H^+\subset \R^n$ is a closed half-space of boundary $H$, (we call $H^-$ the complement of $H^+$) and $|\cdot|$ denotes the Lebesgue measure.

Inequality \eqref{grunbaum} is sharp and equality is attained if and only if $K$ is a cone and $H$ is parallel to the base of $K$. Throughout this text, the term \textit{cone} will refer to the convex hull of a point called its \textit{vertex} and some hyperplanar, compact, convex set called its \textit{base}.

Grünbaum's inequality (GI) has been thoroughly studied since its appearance. Geometrical generalizations have been recently proposed by Stephen and Zhang (\cite{KK}, \cite{JJ}) as well as functional generalizations (see \cite{FF}).

Lying at the heart of convex geometry, GI is commonly deduced from Brunn-Minkowski's inequality (BMI) (see Gardner \cite{BM1}), which, given Lebesgue-measurable sets $A,B\subset \mathbb{R}^n$ such that $A+B$ is also measurable, states that $$|A+B|^{\frac 1 n}\geq |A|^{\frac 1 n}+|B|^{\frac 1 n}$$or, equivalently, for all $\lambda\in [0,1]$,
\begin{equation}\label{grun}
    |\lambda A+(1-\lambda)B|^{\frac 1 n}\geq \lambda |A|^{\frac 1 n}+(1-\lambda)|B|^{\frac 1 n}
\end{equation}
where $\alpha X+\beta Y=\{\alpha x+\beta y:x\in X,y\in Y\}$ is the Minkowski sum of $\alpha X$ and $\beta Y$. BMI also entails the isoperimetric inequality, which, for a sufficiently smooth body $E\subset \mathbb{R}^n$ of perimeter $P(E)$, ensures that
$$P(E)\geq n|E|^\frac{n-1}{n}|B_n|^{\frac 1 n}$$
where $B_n$ is the $n$-dimensional unit ball.
\\

There are at least two natural ways of studying geometric inequalities like these. One is to search for a \textit{reverse} inequality: if some known quantities provide an upper (resp., lower) bound on an unkown quantity, do they also provide a lower (resp., upper) bound on it? A remarkable reverse version of BMI was discovered by Milman \cite{MM}; and one of multiple reverse versions of the isoperimetric inequality can be found in \cite{BB}. Although there aren't any non-trivial reverse versions of GI, we briefly discuss the reversability of our main result in remark \ref{not1}.

Another approach to studying geometric inequalities where equality can be reached begins by determining which bodies yield the equality (sometimes called \textit{optimizers}). At this stage, two questions arise. On one hand, one inquires whether objects that are geometrically close to an optimizer are also close to giving an equality: this is often the case provided that the functions involved in our inequality are in some sense continuous; in this paper, we choose to call such inequalities \textit{consistent}. On the other hand, one tends to prefer the less obvious, reciprocal question: if a body \textit{nearly} yields the equality, is it necessarily geometrically close to an optimizer? This is referred to as studying the \textit{stability} of the inequality. In practice, stability is usually more helpful than consistency in that it allows one to infer complex properties of a body from an often simpler metric. (Obviously, the stability of an inequality depends greatly on the choice of metric for geometrical closeness, which will often be the most natural one.) Interestingly enough, most geometric inequalities display stability despite not being consistent. For a review of the recent advances in the subject of stability, see Figalli's note \cite{Fig1}.

For example, the optimizers of BMI are precisely the pairs $(A,B)$ of homothetic convex sets. One then checks, by summing together Cantor sets, that BMI does not exhibit straightforward consistency if we take, as a measure of the geometrical closeness between measurable sets $X$ and $Y$, the Lebesgue measure of their symmetric difference, $|X\Delta Y|$, where $$X\Delta Y = (X\cup Y)\setminus (X\cap Y).$$
In contrast, BMI does enjoy stability, as was proven by Figalli and Jerison \cite{SS}.

The isoperimetric inequality isn't either consistent, unless, for instance, an upper bound is imposed on the curvature of the boundary. It is, however, stable as proven by Hall, Weitsman and Hayman \cite{HH,HHH}.
\\

It is easy to see that GI is consistent. With $K$ and $H$ defined as above, suppose that $H$ also cuts through the centroid of a cone $C$ whose base is parallel to $H$. Suppose $C$ and $H$ are fixed. Then, since $$|K\Delta C|\geq |(K\cap H^+)\Delta (C\cap H^+)|$$it is clear that
$$\lim _{|K\Delta C|\to 0}\frac{|K\cap H^+|}{|K|}=\frac{|C\cap H^+|}{|C|}.$$

Groemer \cite{Groemer} proved that Grünbaum's inequality is stable. To formulate his result, we define the ``aconicity'' index of $K$,
$$A(K)=\inf\left\{ \frac{|K\Delta C|}{|K|} : C\subset \R^n \text{ is a cone } \right\}.$$
We then have the following stability estimate.

\begin{teor}[Groemer \cite{Groemer}] \label{groe_theo}Let $K$ be a compact and convex subset of $\mathbb{R}^n$, $n\geq 2$, of positive Lebesgue measure and $H$ be a hyperplane passing through the centroid of $K$. If $H^+$ is a half-space of boundary $H$, then
    $$A(K)\leq D(n) \left(\frac{|K\cap H^+|}{|K|}-q_n\right)^{\frac{1}{2n^2}}$$
    for some positive constant $D(n)$ depending only on $n$.
\end{teor}

In our paper, we will show that the exponent $\frac{1}{2n^2}$ can be replaced by $\frac{1}{2n}$.

\begin{teor}\label{maintheo}
    Under the same conditions as in theorem \ref{groe_theo}, we have
$$A(K)\leq 3^{n+7}n^{n+2} \left(\frac{|K\cap H^+|}{|K|}-q_n\right)^{\frac{1}{2n}}.$$
\end{teor}

One should note that, when $n=1$, both theorems are trivially true but uninteresting, due to the fact that if $K\subset\mathbb{R}$ is compact and convex then $K$ \textit{is} a cone, whence $A(K)=0$. Throughout this paper, we shall consider that $n\geq 2$.

\begin{nota}\label{not1}
    One quickly sees that the inequality in theorem \ref{maintheo} is not reversible, that is, no lower bound on $A(K)$ can be given in terms of $\frac{|K\cap H^+|}{|K|}-q_n$. Indeed, given any cone $\mathcal{C}$ and any $\alpha\in[q_n,1-q_n]$, there exists some hyperplane $\Pi$ cutting through the centroid of $\mathcal{C}$ and a half-space $\Pi^+$ of boundary $\Pi$ such that $\frac{|\mathcal{C}\cap \Pi^+|}{|\mathcal{C}|}=\alpha$. Such a $\Pi$ is found by continuously rotating an initial plane, parallel to the base of $\mathcal{C}$ and cutting through its centroid, halfway around and back to itself and arguing about intermediate values.
\end{nota}

\section{Preliminaries}

Given a Lebesgue-integrable function $f:\mathbb{R}\to\mathbb{R}$, its integral over $\mathbb{R}$ will be denoted by $\int f(x)\dx$,  $\int f\dx$ or simply $\int f$; we will specify the integration bounds only when needed.

The purpose of this section is to study the geometry of $K$; this will require defining two central functions and deriving a series of inequalities that will lay the groundwork for section 3, where we will find a cone that provides the desired bound on $A(K)$.

To begin with, we need the following technical lemma.

\begin{lema}\label{intxf}
    Let $f\in L^1(\mathbb{R})$ such that $(x\mapsto xf(x))\in L^1(\mathbb{R})$. Suppose that $\int f\dx=0$ and that there is some $w\in \mathbb{R}$ such that $f(x)\leq 0$ whenever $x<w$ and $f(x)\geq 0$ whenever $x>w$. Then
    $$\int xf(x)\dx\geq 0.$$
    Moreover, if $|f|\leq M$, then $\int |f|\dx\leq 2\sqrt{M\int xf(x)\dx}$.
\end{lema}

\begin{proof}We first see that $f(x)(x-w)\geq 0$ for all $x\in\R$, thus $\int xf(x)\dx=\int (x-w)f(x)\dx\geq 0$. Now, under the assumption that $|f|\leq M$, the second statement will be proven in two steps.

\textit{Step 1)} We will prove that, if $\psi\in L^1([0,+\infty))$ is such that $0\leq \psi\leq M$ and $(x\mapsto x \psi(x))\in L^1([0,\infty))$, then, if $I=\int \psi$, we have $$\int x \psi(x)\dx\geq  \frac{I^2}{2M}.$$ Indeed, let $\phi=M\chi_{[0,I/M]}$. Then $\int \psi=\int\phi$ and $(x-I/M)(\psi(x)-\phi(x))\geq 0$ for all $x\in[0,+\infty)$ thus
$$0\leq \int (x-I/M)(\psi(x)-\phi(x))\dx=\int x(\psi(x)-\phi(x))\dx$$$$=\int x\psi(x)\dx-\int x\phi(x)\dx=\int x\psi(x)\dx-\frac{I^2}{2M}\,$$
as we wanted.

\textit{Step 2)} Since $\int f\dx=0$, we have$$\int xf(x)\dx =\int (x-w)f(x-w)\dx =\int xf(x-w)\dx-w\int f\dx = \int xf(x-w)\dx$$ so we may suppose without loss of generality that $w=0$. Then \begin{align*}
    \int xf(x)\dx&=-\int_{-\infty}^0xf^-(x)\dx+\int_0^{+\infty}xf^+(x)\dx\\
    &=\int_0^{+\infty}xf^-(-x)\dx+\int_0^{+\infty}xf^+(x)\dx\\
    &\geq\frac{(\int f^-)^2}{2M}+\frac{(\int f^+)^2}{2M}.
\end{align*}
The last inequality is justified by applying the property proved in \textit{Step 1)} to functions $f^+$ and $x\mapsto f^-(-x)$, which are both positive and, since $w=0$, supported on $[0, +\infty)$. Now, using that $\int f^+\dx=\int f^-\dx=\frac 12\int |f|\dx$, we arrive at $\int xf(x)\dx\geq \frac{(\int |f|)^2}{4M}$, therefore proving the lemma.\end{proof}

\begin{defi}
    Given any set $S\subset \mathbb{R}^n$ and $x\in\mathbb{R}$, we define $$S_x=\{y\in\mathbb{R}^{n-1}:(x,y)\in S\}.$$
\end{defi}

We will, at this point, fix $K$ and $H$, but we want to restrict ourselves to the case where $K$ and $H$ have a particular shape that will considerably simplify our reasoning. More precisely, we will suppose that
\begin{enumerate}
    \item [(H1)] the centroid of $K$ is located at the origin;
    \item [(H2)] $H=\{0\}\times \mathbb{R}^{n-1}$;
    \item [(H3)] $\max_{x\in\mathbb{R}}|K_x|=1$; and
    \item [(H4)] $|K|=1$.    
\end{enumerate}
Given any compact, convex body $\tilde K$ of positive Lebesgue measure and any hyperplane $\tilde H$ cutting through the centroid of $\tilde K$, there exists an affine bijection $f$ such that $f(\tilde K)$ and $f(\tilde H)$ satisfy the above hypotheses. To see why, simply observe that (H1) is obtained by translation, then (H2) by any linear transformation sending $\tilde H$ to $\{0\}\times\mathbb{R}^{n-1}$, then (H3) by a suitable rescaling and finally (H4) by stretching or compressing the space along the first axis.

We claim that the above assumptions can be made without loss of generality, that is, that the inequality in theorem \ref{maintheo} holds for the pair $(\tilde K,\tilde H^+)$ if and only if it holds for the pair $(f(\tilde K), f(\tilde H^+))$. This is simply due to the fact that
$$\frac{|f(\tilde K)\cap f(\tilde H^+)|}{|f(\tilde K)|}=\frac{|\det(f)||\tilde K\cap \tilde H^+|}{|\det (f)||\tilde K|}=\frac{|\tilde K\cap \tilde H^+|}{|\tilde K|}$$
and that
$$A(f(\tilde K))=\inf_{C \text{ cone}}\frac{|f(\tilde K)\Delta C|}{|f(\tilde K)|}=\inf_{C \text{ cone}}\frac{|f(\tilde K)\Delta f(C)|}{|f(\tilde K)|}=\inf_{C \text{ cone}}\frac{|\tilde K\Delta C|}{|\tilde K|}=A(\tilde K)$$
because $f$ induces a bijection among all the cones contained in $\mathbb{R}^n$. We thus fix a pair $(K,H)$ that meets the hypotheses (H1) through (H4) and will show theorem \ref{maintheo} for this special case, effectively proving it for the general case as well.

\begin{defi}\label{defi_n_geq_2}
    Let $a=\min \{x_1:(x_1,\ldots,x_n)\in K\}$ and $b=\max \{x_1:(x_1,\ldots,x_n)\in K\}$. We define the function $g:[a,b]\to \mathbb{R}$ as $$g(x)=|K_x|^\frac 1{n-1}.$$
\end{defi}
\begin{nota}
    $g$ is concave (in particular, continuous); to see why, observe that $K_{\lambda x + (1-\lambda)y}\supset \lambda K_x+(1-\lambda)K_y$ which entails
    
    \begin{align*}
        g(\lambda x + (1-\lambda)y)&=|K_{\lambda x + (1-\lambda)y}|^{\frac 1 {n-1}}\\
        &\geq |\lambda K_x+(1-\lambda)K_y|^{\frac 1 {n-1}}\\
        &\geq \lambda |K_x|^{\frac 1{n-1}}+(1-\lambda) |K_y|^{\frac 1{n-1}}\\
        &=\lambda g(x)+(1-\lambda)g(y)
    \end{align*}
    where we have used Brunn-Minkowski's inequality in the form given in \eqref{grun}.
\end{nota}

\begin{lema}\label{abbounds} The following inequalities hold.
\begin{equation*}
    \begin{cases}
        b-a\leq n\\
        -n<a<-\frac13<\frac13 < b < n
    \end{cases}
\end{equation*}
\end{lema}
\begin{proof}
    By hypothesis (H3), we may find $x_1\in[a,b]$ such that $g(x_1)=1$. Letting $S_a\in\{a\}\times K_a$ and $S_b\in\{b\}\times K_b$ be arbitrary, the cones $\mathrm{Conv}(S_a,\{x_1\}\times K_{x_1})$ and $\mathrm{Conv}(S_b,\{x_1\}\times K_{x_1})$ are contained in $K$ (where $\mathrm{Conv}(\cdot,\cdot)$ denotes the convex hull). Using that $|K_{x_1}|=g(x_1)^{n-1}=1$, the volume of the union of these two cones is clearly
    $$|K_{x_1}|\frac{b-x_1}{n}+|K_{x_1}|\frac{x_1-a}{n}=\frac{b-a}{n}.$$
    This quantity being less than $|K|=1$, we have shown that $b-a\leq n$.
    
    Moreover, $b=\int_0^b\dx\geq \int_0^bg(x)^{n-1}\dx=|K\cap H^+|\geq q_n>\frac 13$ and similarly $a<-\frac13$. Finally $b\leq n+a<n$ and $a\geq -n+b>-n$.
\end{proof}
\begin{lema}\label{k0bound}
    The inequality $$\frac 1 {3n}\leq |K_0|\leq 1$$ holds true.
\end{lema}
\begin{proof}
We start by noticing that $|K_0|=g(0)^{n-1}\leq 1$. Moreover, by the properties of concave functions, $g(x)\leq g(0)$ holds either for all $x\in[0,b]$ or for all $x\in[a,0]$. In the first case, we have $$q_n\leq \int _0^bg(x)^{n-1}\dx\leq bg(0)^{n-1}\leq n|K_0|\implies |K_0|\geq q_n/n$$and therefore $|K_0|\geq \frac 1 {3n}$. The same inequality follows from the second case analogously.
\end{proof}

\begin{nota}
    Lemmas \ref{abbounds} and \ref{k0bound} are crucial to our reasoning because they provide nontrivial bounds for $|K_0|$, $|a|$ and $|b|$. The existence of such general bounds is essentially guaranteed by hypothesis (H3). To be more precise, assumptions (H1), (H2) and (H4) are not sufficient to derive bounds on $|a|$, $|b|$ and $|K_0|$. Indeed, it may happen that $|K_0|$ is arbitrarily small while $|a|$ and $|b|$ are arbitrarily large; just as $|K_0|$ may be arbitrarily large while $|a|$ and $|b|$ become be arbitrarily small.

\end{nota}

\begin{defi}
    With $b'=\frac{n}{|K_0|}|K\cap H^+|>0$ and $a'=b'-\left(n(b')^{n-1}/|K_0|\right)^{1/n}$, define $c:[a',b']\to\R$ by
        $$c(x)=g(0)\frac{b'-x}{b'}.$$
\end{defi}

\begin{prop}\label{propdefi2}
    The following properties hold.
    \begin{enumerate}
    \item $a'< 0$
        \item $c(0)=g(0)$
        \item $\int_{a'}^0c(x)^{n-1}\dx=|K\cap H^-|$
        \item $\int_0^{b'}c(x)^{n-1}\dx=|K\cap H^+|$
        \item $\int_{a'}^{b'} c^{n-1}=1$
    \end{enumerate}
\end{prop}
\begin{proof}
     We prove each property in turn.\\
     (1) Observe that $$b'=\frac{n}{|K_0|}|K\cap H^+|<n/|K_0|$$ whence
     \begin{align*}
     (b')^n&< n(b')^{n-1}/|K_0|\\
     b'&<(n(b')^{n-1}/|K_0|)^{1/n}
     \end{align*} and therefore $a'< 0$.\\
     (2) This is immediate from the definition of function $c$.\\
     (4) We have $$\int_0^{b'}c(x)^{n-1}\dx=\frac{(g(0))^{n-1}}{(b')^{n-1}}\frac{(b')^n}{n}=\frac{|K_0|b'}{n}=|K\cap H^+|.$$
     (5) Observe that $$\int_{a'}^{b'} c^{n-1}=\frac{(g(0))^{n-1}}{(b')^{n-1}}\frac{(b'-a')^{n}}{n}=\frac{|K_0|}{(b')^{n-1}}\frac{n(b')^{n-1}/|K_0|}{n}=1.$$
     (3) This follows from $(4)$ and $(5)$:
     $$\int_{a'}^0c(x)^{n-1}\dx=\int_{a'}^{b'}c(x)^{n-1}\dx-\int_0^{b'}c(x)^{n-1}\dx= 1-|K\cap H^+|=|K\cap H^-|.$$
\end{proof}

\begin{nota}
    To better understand the role of function $c$, consider $C_K$ a cone such that $|C_K\cap H^+|=|K\cap H^+|$, $|C_K\cap H^-|=|K\cap H^-|$ and $(C_K)_0=K_0$. Then $C_K$ is unique up to shear along $H$, that is, up to linear transformation fixing $H$ and of determinant $1$. Moreover, $c^{n-1}$ gives the area of each cross section of $C_K$ that is parallel to $H$. Figure \ref{figure} shows the graphs of $g$ and $c$ as they would appear in a general setting. One could hope to find a $C_K$ such that $|K\Delta C_K|=\int_{\mathbb{R}} |c^{n-1}-g^{n-1}|\dx$, which is small, as shown in proposition \ref{bprimebbound}. However, we do not see any reason why such a $C_K$ would exist so, instead, in section 3 we show that a cone $C$ can be found whose cross-sectional area is close to that of $C_K$ while $|K\Delta C|$ is easy to calculate.
\end{nota}

\begin{prop}\label{aborder}
    $a\leq a'$ and $b\leq b'$. Moreover, $g(x)\leq c(x)$ for all $x\in [a',0]$ and there exists $v\in (0,b]$ such that $g(x)\geq c(x)$ for every $x\in [0,v]$ and $g(x)\leq c(x)$ for all $x\in [v,b]$.
\end{prop}

\begin{proof} If instead it were $b'<b$, then by the concavity of $g$ we would have that $g(x)>c(x)$ for all $0<x\leq b'$ implying that $\int_0^{b'}c(x)^{n-1}\dx<\int_0^{b}g(x)^{n-1}\dx=|K\cap H^+|$, a contradiction. Thus $b\leq b'$. Moreover, there must exist some $\beta\in (0,b]$ such that $g(\beta)\geq c(\beta)$; otherwise, condition $(4)$ in proposition \ref{propdefi2} would not hold, since we would have that
$$\int_0^{b'}c^{n-1}\geq \int_0^{b}c^{n-1}>\int_0^{b}g^{n-1}=|K\cap H^+|.
$$Moreover, the set $\{x\in[0,b]:g(x)\geq c(x)\}$ is closed, convex and contains $0$ and $\beta$, hence of the form $[0,v]$ for some $v\in [\beta,b]$.

Suppose now $a'<a$. The concavity of $g$ implies $g(x)\leq c(x)$ for all $a\leq x\leq 0$. To see why, suppose $g(z)>c(z)$ for some $z\in[a,0]$; then we would have \begin{align*}
    g(0)=g\left(z\frac{\beta}{\beta-z}+\beta\frac{-z}{\beta-z}\right)\geq \frac{\beta}{\beta-z}g(z)  -\frac{z}{\beta-z}g(\beta)&>\frac{\beta}{\beta-z}c(z)  -\frac{z}{\beta-z}c(\beta)\\&=c(0)=g(0)
\end{align*} which is a contradiction. Hence $\int_{a'}^0c(x)^{n-1}\dx>\int_{a}^0c(x)^{n-1}\dx\geq \int_{a}^0g(x)^{n-1}\dx=|K\cap H^-|$, a contradiction.

Thus $a\leq a'$. An analogous argument shows that $c(x)\geq g(x)$ for all $x\in [a',0]$.

\end{proof}

Moreover lemma \ref{k0bound} implies that $b'/(3n^2)\leq b'|K_0|/n=|K\cap H^+|\leq 1$, hence
\begin{equation}\label{bprimebound}
    b'\leq 3n^2.
\end{equation}

\begin{figure}
    \centering
    \begin{tikzpicture}[scale=1.7]

      \definecolor{mygreen}{RGB}{0,200,40}
      \definecolor{myblue}{RGB}{20,50,200}
      \definecolor{mygray}{RGB}{0,0,0}

      \draw[line width=0.8pt] (-3,0) -- (4,0); 
      \draw[line width=0.8pt] (0,-0.6) -- (0,3); 

      \def\a{-2.1}
      \def\ap{-1.3}
      \def\b{1.66}
      \def\bp{3.1}

      \def\g#1{-0.32*#1*#1 + 0.17*#1 + 2}
      \def\c#1{-2/\bp*#1 + 2}

      \draw[black, line width=1pt] (\a,0.07) -- (\a,-0.07);
      \node[mygray, below] at (\a,-0.12) {$a$};
      \foreach \x/\name in {\ap/$a'$, \b/$b$, \bp/$b'$} {
        \draw[black, line width=1pt] (\x,0.07) -- (\x,-0.07);
        \node[mygray, below] at (\x,-0.05) {\name};
      }
      
      \draw[dotted, line width=1pt] (\ap,{\c{\ap}}) -- (\ap,0);
      \draw[dotted, line width=1pt] (\a,{\g{\a}}) -- (\a,0);
      \draw[dotted, line width=1pt] (\b,{\g{\b}}) -- (\b,0);

      \draw[myblue, line width=1.4pt, domain=\a:\b, samples=200] plot (\x, {\g{\x}});
      \node[myblue] at (-1.5,1.3) {$g$};

      \draw[mygreen, line width=1.4pt, domain=\ap:\bp, samples=2] plot (\x, {\c{\x}});
      \node[mygreen] at (2,0.91) {$c$};

      \fill[myblue] (\a, {\g{\a}}) circle (1.2pt);
      \fill[myblue] (\b, {\g{\b}}) circle (1.2pt);
      \fill[mygreen] (\ap, {\c{\ap}}) circle (1.2pt);
      \fill[mygreen] (\bp, {\c{\bp}}) circle (1.2pt);

    \end{tikzpicture}
    \caption{The graphs of functions $g$ and $c$ in a general setting. Note that $g(a)$ and $g(b)$ may be zero; otherwise, it means that $K$ has ``flat faces''.}
    \label{figure}
\end{figure}

\begin{nota}
    From this point on, we will extend functions $g$ and $c$ by $0$ outside their original domains.
\end{nota}

\section{Proof of the main theorem}

The purpose of this last section is to find a cone $C$ such that $|K\Delta C|$ both small and easy to calculate. Groemer found one such cone to lie \textit{inside} $K$, which made the computation of $|K\Delta C|$ trivial. We will exploit the convex nature of $K$ to find a cone that neither contains nor is contained by $K$ but nevertheless suits the above requirements.

The first thing we would like to do is to prove that $b'$ is close to $b$ provided that $|K\cap H^+|$ is close to $q_n$. We make the following claim.

\begin{prop}\label{bprimebbound}
Let $d=|K\cap H^+|^{1/n}-q_n^{1/n}$. Then $b'-b\leq 288n^2 d^{1/(2n)}$.
\end{prop}

\begin{proof} 
Let $h=c^{n-1}-g^{n-1}$. The key is to observe that
\begin{equation}\label{remember}
    \int |h|\geq \int_b^{b'}|h|=\int_{b}^{b'}c(x)^{n-1}\dx=|K\cap H^+|\left(\frac{b'-b}{b'}\right)^n\geq \frac 13\left(\frac{b'-b}{b'}\right)^n.
\end{equation}
We would like to give a bound on $\int |h|$, so we start by observing that
$$\int xc(x)^{n-1}\dx=(b'-a')d.$$
To prove this fact, notice first that $$
\frac{b'}{b'-a'}=\frac{b'}{(n(b')^{n-1}/|K_0|)^{1/n}}=\left(\frac{b'|K_0|}{n}\right)^{1/n}=|K\cap H^+|^{1/n}.$$ Since the centroid of an $n$-dimensional cone of height $t$ is located in the hyperplane that lies $t n/(n+1)$ away from its vertex and $t/(n+1)$ from its base, we get
$$\int xc(x)^{n-1}\dx=a'+(b'-a')\frac{1}{n+1}=(b'-a')\left(\frac{b'}{b'-a'}-\frac{n}{n+1}\right)=(b'-a')d.$$

Therefore, by lemma \ref{abbounds} and equation \eqref{bprimebound}, $$\int xc(x)^{n-1}\dx= (b'-a')d\leq (3n^2+n)d\leq 4n^2d.$$

Note that, by the definition of $h$ and since $\int xg(x)^{n-1}\dx=0$, we have
\begin{equation}\label{intxhbound}
    \int xh(x)\dx\leq 4n^2d.
\end{equation}
We now write $h=h_1+h_2$ where $h_1=h\chi_{(-\infty,0)}$ and $h_2=h\chi_{[0,+\infty)}$. Proposition \ref{propdefi2} readily implies that $\int h_1=\int h_2=0$.

According to proposition \ref{aborder}, we can apply lemma \ref{intxf} to $h_1$ as well as to $h_2$, thus, by equation \eqref{intxhbound},
\begin{equation}\label{intxhibound}
    0\leq \int xh_i(x)\dx\leq 4n^2d,\quad i=1,2.
\end{equation}
Now, we would like to check that $|h|$ is bounded by some constant that does not depend on $K$. We know that $g^{n-1}\leq 1$ and that $c^{n-1}\leq c^{n-1}(a')$. But $$c^{n-1}(a')(b'-a')/n=1\implies c^{n-1}(a')=\frac{n}{b'-a'}\leq \frac{n}{1/3},$$ because of lemma \ref{abbounds}. Hence, $c^{n-1}\leq 3n$ and thus
\begin{equation}\label{abshbound}
    |h_i|\leq |h|\leq g^{n-1}+c^{n-1}\leq 3n+1,\quad i=1,2.
\end{equation}
Thus, as a consequence of \eqref{intxhibound}, \eqref{abshbound} and lemma \ref{intxf},
$$\int |h_i|\leq 2\sqrt{(3n+1)\cdot 4n^2d},\quad i=1,2$$
which implies that
\begin{equation}\label{inthbound}
    \int |h| = \int |h_1|+\int |h_2|\leq 4\sqrt{(3n+1)4n^2d}\leq 16n^{3/2}\sqrt{d}.
\end{equation}
On the other hand, by \eqref{remember},
$$\int |h|\geq \frac 13\left(\frac{b'-b}{b'}\right)^n,$$
hence, by \eqref{inthbound}, $\frac 13\left(\frac{b'-b}{b'}\right)^n\leq 16n^{3/2}\sqrt{d}$ and, since $b'\leq 3n^2$, we obtain that
$$b'-b\leq 3n^2(48n^{3/2}\sqrt{d})^{1/n}\leq 288n^2 d^{1/(2n)}$$
since $(n^{3/2})^{1/n}\leq 2$.\end{proof}

\begin{defi}
    Let $\beta\in K_b$ be arbitrary and $S=(b,\beta)\in K$. Let
    
    $$C=\text{Conv}\Bigg[\{a'\}\times \left(\beta + \frac{b-a'}{b}(K_0-\beta)\right), \,S\Bigg].$$
    For $x\in \R$, let $$C_x=\{y\in \mathbb{R}^{n-1}:(x,y)\in C\}$$ and $s(x)=|C_x|^\frac 1{n-1}$.
\end{defi}
Adding the graph of function $s$ to figure \ref{figure} results in figure \ref{figure2}.
\begin{prop}
    For $x\in [a', b]$, we have that$$C_x = \beta + \frac{b-x}{b}(K_0-\beta).$$
    In particular, $C_0=K_0$ and $s(x)=g(0)\frac{b-x}{b}$.
\end{prop}
\begin{proof}
    We have $$C_x=\left(\beta + \frac{b-a'}{b}(K_0-\beta)\right)\frac {b-x}{b-a'}+\beta \frac {x-a'}{b-a'}=\beta + \frac{b-x}{b}(K_0-\beta).$$
\end{proof}
\begin{prop}
    $|K\Delta C|=\int |g^{n-1}-s^{n-1}|$.
\end{prop}

\begin{proof}
On one side, by the convexity of $K$, we have that
$$C\cap H^+=\text{Conv}\big[\{0\}\times C_0,\,S\big]=\text{Conv}\big[\{0\}\times K_0,\,S\big]\subset K\cap H^+.$$

On the other side,
$$C\cap ([a',0]\times \R^{n-1})\supset K\cap ([a',0]\times \R^{n-1}).$$
To see why the last containment holds, let $P=(x,y)\in K\cap ([a',0]\times \R^{n-1})$. Then $Q:=\frac{b'}{b'-x}P+\frac{-x}{b'-x}S\in \{0\}\times K_0=\{0\}\times C_0 \subset C$. Since $P$ lies on the line $(SQ)$ with $S,Q\in C$ and $a'\leq x\leq b$, necessarily $P\in C$.

Hence, we can now write
\begin{align*}|K\Delta C|&=|(K\Delta C)\cap H^+|+|(K\Delta C)\cap ([a',0]\times \R^{n-1})|+|(K\Delta C)\cap ((-\infty,a')\times \R^{n-1})|\\
&=|(K\setminus C)\cap H^+|+|(C\setminus K)\cap ([a',0]\times \R^{n-1}))|+|K\cap ((-\infty,a')\times \R^{n-1})|\\
&=\int_0^\infty (g^{n-1}-s^{n-1})+\int_{a^-}^0(s^{n-1}-g^{n-1})+\int_{-\infty}^{a^-}g^{n-1}\\
&=\int_{0}^{\infty}|g^{n-1}-s^{n-1}|+\int_{a^-}^0|g^{n-1}-s^{n-1}|+\int_{-\infty}^{a^-} |g^{n-1}-s^{n-1}|\\&=\int |g^{n-1}-s^{n-1}|.
\end{align*}
\end{proof}
\begin{figure}
    \centering
    \begin{tikzpicture}[scale=1.7]

      \definecolor{mygreen}{RGB}{0,200,40}
      \definecolor{myblue}{RGB}{20,50,200}
      \definecolor{mygray}{RGB}{0,0,0}
      \definecolor{myred}{RGB}{200,0,0}

      \draw[line width=0.8pt] (-3,0) -- (4,0); 
      \draw[line width=0.8pt] (0,-0.6) -- (0,3.7); 

      \def\a{-2.1}
      \def\ap{-1.3}
      \def\b{1.66}
      \def\bp{3.1}

      \def\g#1{-0.32*#1*#1 + 0.17*#1 + 2}
      \def\c#1{-2/\bp*#1 + 2}
      \def\s#1{2*(\b-#1)/\b}

      \draw[black, line width=1pt] (\a,0.07) -- (\a,-0.07);
      \node[mygray, below] at (\a,-0.12) {$a$};
      \foreach \x/\name in {\ap/$a'$, \b/$b$, \bp/$b'$} {
        \draw[black, line width=1pt] (\x,0.07) -- (\x,-0.07);
        \node[mygray, below] at (\x,-0.05) {\name};
      }

      \draw[dotted, line width=1pt] (\ap,{\s{\ap}}) -- (\ap,0);
      \draw[dotted, line width=1pt] (\a,{\g{\a}}) -- (\a,0);
      \draw[dotted, line width=1pt] (\b,{\g{\b}}) -- (\b,0);

      \draw[myblue, line width=1.4pt, domain=\a:\b, samples=200] plot (\x, {\g{\x}});
      \node[myblue] at (-1.5,1.3) {$g$};

      \draw[mygreen, line width=1.4pt, domain=\ap:\bp, samples=2] plot (\x, {\c{\x}});
      \node[mygreen] at (2,0.91) {$c$};

      \draw[myred, line width=1.4pt, domain=\ap:\b, samples=2] plot (\x, {\s{\x}});
      \node[myred] at (0.9,0.65) {$s$};

      \fill[myblue] (\a, {\g{\a}}) circle (1.2pt);
      \fill[myblue] (\b, {\g{\b}}) circle (1.2pt);
      \fill[mygreen] (\ap, {\c{\ap}}) circle (1.2pt);
      \fill[mygreen] (\bp, {\c{\bp}}) circle (1.2pt);
      \fill[myred] (\ap, {\s{\ap}}) circle (1.2pt);
      \fill[myred] (\b, {\s{\b}}) circle (1.2pt);

    \end{tikzpicture}
    \caption{The graphs of functions $g$, $c$ and $s$ in a general setting (to be exact, the graphs of their restriction to their respective supports; one should keep in mind that the functions are defined on all $\mathbb{R}$).}
    \label{figure2}
\end{figure}
A final proposition is needed before proving the main theorem.

\begin{prop}\label{intcsbound}
    $\int |c^{n-1}-s^{n-1}|\leq64\cdot 3^{n+2}n^{n+2}d^{1/(2n)}$.
\end{prop}

\begin{proof}
We first observe that
\begin{align*}
    s(a')-c(a')&=g(0)\left(\frac{b-a'}{b}-\frac{b'-a'}{b'}\right)\leq \frac{b-a'}{b}-\frac{b'-a'}{b'}\\
    &\leq  -a\left(\frac{1}{b}-\frac{1}{b'}\right)\leq n\frac {b'-b}{bb'}\leq 2592n^3d^{1/(2n)}\stepcounter{equation}\tag{\theequation}\label{scbound}
\end{align*}

where we have used lemma \ref{abbounds} in the last two inequalities. Moreover,

\begin{equation}\label{sbound}
    c(a')\leq s(a')= g(0)\frac{b-a'}{b}\leq \frac{b-a}{1/3}\leq3n.
\end{equation}
Now, obviously $c(x)\leq s(x)$ holds for $x\leq 0$ and $c(x)\geq s(x)$ holds for $x\geq 0$. Thus
\begin{align*}
    \int |c^{n-1}-s^{n-1}|&=\int_{a'}^0 (s^{n-1}- c^{n-1})+\int_0^{b'}(c^{n-1}-s^{n-1})\\
    &=\int s^{n-1}-\int c^{n-1}+2\int_0^{b'}c^{n-1}-2\int_0^{b'}s^{n-1}\\
&=s^{n-1}(a')\frac{b-a'}{n}-c^{n-1}(a')\frac{b'-a'}{n}+2|K_0|\frac{b'-b}{n}\\
&=\frac{b-a'}{n}(s^{n-1}(a')-c^{n-1}(a'))+\frac{b'-b}{n}(2|K_0|-c^{n-1}(a')).\stepcounter{equation}\tag{\theequation}\label{lasteq}\end{align*}

Moreover, by equations \eqref{scbound} and \eqref{sbound}, \begin{align*}
s^{n-1}(a')-c^{n-1}(a')&=(s(a')-c(a'))(s^{n-2}(a')+\ldots+c^{n-2}(a'))\\
&\leq 2592n^3d^{1/(2n)}(n-1)(3n)^{n-2}\leq 32\cdot 3^{n+2}n^{n+2}d^{1/(2n)}.
\end{align*}

Finally, $$2|K_0|-c^{n-1}(a')\leq 2|K_0|+c^{n-1}(a')\leq 2+(3n)^{n-1}\leq 3^{n}n^{n-1}.$$

Thus, substituting in \eqref{lasteq}, and making use of lemma \ref{abbounds} and proposition \ref{bprimebbound}, we get \begin{align*}\int |c^{n-1}-s^{n-1}| &\leq \frac{n}{n}\big(32\cdot 3^{n+2}n^{n+2}d^{1/(2n)}\big)+ \frac{288n^2d^{1/(2n)}}{n}\big(3^{n}n^{n-1}\big)\\ &\leq 64\cdot 3^{n+2}n^{n+2}d^{1/(2n)}.\end{align*}\end{proof}
We can finally prove our main result.

\begin{proof}[Proof of theorem \ref{maintheo}]

First,
\begin{align*}
    A(K)\leq |K\Delta C|&=\int |g^{n-1}-s^{n-1}|\\
    &\leq \int |h|+\int|c^{n-1}-s^{n-1}|\\
    &\leq 16n^{3/2}\sqrt{d} +64\cdot 3^{n+2}n^{n+2}d^{1/(2n)},
\end{align*} the last step following from \eqref{inthbound} and proposition \ref{intcsbound}. Since $d\leq 1$, we have $\sqrt{d}\leq d^{1/(2n)}$ and thus
$$A(K)\leq 80\cdot 3^{n+2}n^{n+2}d^{1/(2n)}\leq 3^{n+6}n^{n+2}d^{1/(2n)}.$$
Now, $d=t^{1/n}-q_n^{1/n}$ where $t=|K\cap H^+|$. By the mean value theorem, there exists some $\xi\in [q_n,t]$ such that $d=\frac 1 n\xi^\frac{1-n}{n} (t-q_n)$ and, because $\xi\geq \frac 13$, we have that $\xi^{\frac{1-n}{n}}\leq 3$ thus
$$d\leq 3(t-q_n)$$
which in turn implies that
$$A(K)\leq 3^{n+6}n^{n+2}(3(t-q_n))^{1/(2n)}\leq 3^{n+7}n^{n+2}(t-q_n)^{1/(2n)}.$$

\end{proof}

\end{document}